\documentclass{amsart}
\usepackage[english]{babel}
\usepackage{graphicx}
\usepackage{amsmath}
\usepackage{amsthm}
\usepackage{amssymb}
\newtheorem{theorem}{Proposition}[section]
\newtheorem{theorem 2}[theorem]{Theorem}
\newtheorem{theorem 3}[theorem]{Theorem}
\newtheorem{theorem 4}[theorem]{Theorem}
\newtheorem{theorem 5}[theorem]{Lemma}
\newtheorem{theorem 6}[theorem]{Lemma}
\newtheorem{theorem 7}{Lemma}[section]
\newtheorem{corollary}[theorem]{Corollary}
\newtheorem{corollary 1}[theorem]{Corollary}

\begin{document}
\title{On the coefficients of divisors of $x^{n}-1$}
\address{Department of Mathematics, Indian Institute of Technology Roorkee,India 247667}
\author{Sai Teja Somu}
\date{October 21, 2015}
\maketitle
\begin{abstract}
 Let $a(r,n)$ be $r$th coefficient of $n$th cyclotomic polynomial. Suzuki proved that $\{a(r,n)|r\geq 1,n\geq 1\}=\mathbb{Z}$. If $m$ and $n$ are two natural numbers we prove an analogue of Suzuki's theorem for divisors of $x^n-1$ with exactly $m$ irreducible factors. We prove that for every finite sequence of integers $n_1,\ldots,n_r$ there exists a divisor $f(x)=\sum_{i=0}^{deg(f)}c_ix^i$ of $x^n-1$ for some $n\in \mathbb{N}$ such that $c_i=n_i$ for $1\leq i \leq r$. Let $H(r,n)$ denote the maximum absolute value of $r$th coefficient of divisors of $x^n-1$.
 In the last section of the paper we give tight bounds for $H(r,n)$.
\end{abstract}
\section{Introduction}
The $n$th cyclotomic polynomial($\phi_n(x)$) is given by
\begin{equation}\label{equation}
\phi_n(x)=\prod\limits_{d|n}(x^d-1)^{\mu(\frac{n}{d})},
\end{equation}
where $\mu(n)$ is the Mobius function.
 $\phi_n(x)$ is an irreducible polynomial of degree $\phi(n)$. $x^n-1$ can be factored in the following way
\begin{equation}
x^n-1=\prod\limits_{d|n}\phi_d(x).
\end{equation}
Let $A(n)$ denote the largest coefficient of $\phi_n(x)$.  Bateman in \cite{B} proved the following inequality 
\begin{equation}
A(n)\leq n^{2^{k-1}},
\end{equation}
where $k$ is the number of distinct odd prime factors. $A(n)$ has been investigated in papers \cite{C},\cite{D} and \cite{E}. For a polynomial $f\in \mathbb{Z}[x]$, let $H(f)$ denote the absolute value of largest coefficient of $f$. Pomerance and Ryan in \cite{F} introduced the function
\begin{equation}
B(n):=max\{H(f) : f|(x^n-1)\}.
\end{equation}
They obtained a tight estimate for $B(n)$ and proved that 
\begin{equation}
\limsup_{n\to\infty}\frac{\log \log B(n)}{\log n/\log \log n}=\log 3.
\end{equation}
Let $d(n)$ denote number of divisors of $n$.
Thompson in Theorem 1.2 of \cite{Z} has proven that for any function $\psi(n)$ defined on natural numbers such that $\psi(n)\rightarrow \infty$ as $n\rightarrow  \infty$ we have
\begin{equation}
B(n)\leq n^{d(n)\psi(n)}
\end{equation}
for a set of natural numbers of density $1$. 

The function $B(n)$ has been the subject of  the papers \cite{G} and \cite{H}.
If $\phi_n(x)=\sum\limits_{r=0}^{\phi(n)}a(r,n)x^r$, Suzuki in \cite{A} proves that $\{a(r,n)|r\geq 1,n\geq 1\}=\mathbb{Z}$.
We prove an analogue for divisors of $x^n-1$ in the second section.
In the third section we prove that for every  finite sequence of integers
$n_1,\ldots,n_r$ there exists a divisor $f(x)=\sum_{i=0}^{deg(f)}c_ix^i$ of $x^n-1$ for some $n\in \mathbb{N}$ such that $c_i=n_i$ for $1\leq i \leq r$.
 
 For $f\in \mathbb{Z}[x]$, let $(f)_r$ denote the $r$th coefficient of $f$.
 Then $H(r,n)$ is defined in the following way
 \begin{equation}
 H(r,n)=max\{|(g)_r|: g|(x^n-1)\}.
 \end{equation}
 In the last section of the paper we give an upper bound for $H(r,n)$ in terms of number of divisors of $n$ and then we show that the upper bound cannot be improved significantly.
   
\section{An analogue of Suzuki's Theorem}
   
 If $m$ and $n$ are two natural numbers then there exists a divisor $f(x)$ of $x^l-1$ for some $l\in \mathbb{N}$ with exactly $m$ irreducible factors such that $(f)_r=n$ and $(f)_s=-n$ for some $r,s\in \mathbb{N}$. We prove the following proposition in this section.
   
\begin{theorem} \label{Theorem 1}
Let $n$ and $m$ be two natural numbers then there exists a divisor $f(x)=\sum\limits_{i=0}^{deg(f)}a_ix^i$ of $x^{l}-1$ for some $l\in \mathbb{N}$ with exactly $m$ irreducible divisors 
 such that $\{-n,\cdots,0,\cdots,n\}\subset\{a_1,\cdots,a_{deg(f)}\}$.
\end{theorem}  

We follow the proof of Suzuki \cite{A}.
We require the following lemmas for proving Theorem \ref{Theorem 1}.
\begin{theorem 5}\label{lemma 1}
Let $t$ be any integer greater than 2. Then there exist $t$ distinct primes $p_1<p_2<\cdots<p_t$ such that $p_1+p_2>p_t$.
\end{theorem 5}
\begin{proof}
Let $\pi(n)$ be prime counting function.
From the prime number therem one can see that $\lim_{n\to\infty}(\pi(2n)-\pi(n))=\infty.$
Hence for a large $n$ we can choose $t$ distinct primes between $(n,2n]$ and the $t$ primes picked from the set $(n,2n]$ clearly satisfy the hypothesis.
\end{proof}
\begin{theorem 6}\label{Lemma 2}
If $p$ is a prime.
\newline
If $p|n$ then $\phi_{np}(x)=\phi_n(x^p)$.
\newline
If $p\nmid n$ then $\phi_{np}(x)=\frac{\phi_n(x^p)}{\phi_n(x)}$.
\newline
If $n>1$ is odd then $\phi_{2n}(x)=\phi_n(-x)$.
\end{theorem 6}
\begin{proof}
All the three statements can be proved from (\ref{equation}).
\end{proof}
Proof of Proposition \ref{Theorem 1}
\begin{proof}
Choose an odd $t$ strictly greater than $n+1$ and from Lemma \ref{lemma 1} there exist $t$ primes $p_1<\cdots<p_t$ such that $p_1+p_2>p_t$. Let $p=p_t$ and $N=\prod\limits_{i=1}^{t}p_i$.

From (1), we have
\begin{align*}
\phi_{N}(x)&=\prod_{d|N}(1-x^d)^{\mu(\frac{N}{d})}\\
&\equiv (1-x)^{-1}\prod_{i=1}^{t}(1-x^{p_i})(\mod x^{p+1})\\
&\equiv (1+x+\cdots x^{p})(1-x^{p_1}\cdots-x^{p_t})(\mod x^{p+1})\\
&\equiv \sum_{i=0}^{p}c_ix^i (\mod x^{p+1}),
\end{align*}
where $c_i=\left\{\begin{matrix}
 & 1  & 0 \leq i <p_1   \\ 
 & 1-k  & p_k\leq i<p_{k+1}\\ 
 & 1-t &  i=p.
\end{matrix}\right.$\newline
\newline
From Lemma \ref{Lemma 2}, we have
\begin{align*}
\phi_{2N}(x)&=\phi_{N}(-x)\\
&\equiv \sum_{i=0}^{p}c_i(-1)^ix^i (\mod x^{p+1}).
\end{align*}
We can see that 
$\{-n,\cdots,0,\cdots,n\}\subset\{-c_1,\cdots,c_p(-1)^p\}.$
Hence Proposition \ref{Theorem 1} is true for $m=1$.
Choose a prime $p'$ greater than $p$, from Lemma \ref{Lemma 2}, we have
 \begin{align*}
 \phi_{2Np'}(x)\phi_{p'}(x)=\phi_{2N}(x^{p'}).
 \end{align*}
 Hence the set $\{-n,\cdots,n\}$ is a subset of set of coefficients of $\phi_{2Np'}(x)\phi_{p'}(x)$.
 Hence the proposition is true for $m=2$.
Choose square free numbers $n_1<n_2<\cdots<n_k$ and $n'_1<\cdots<n'_k$ such that each $n_i$ has exactly two distinct prime factors and each $n'_i$ are prime and every prime factor of $n_i$ and $n'_i$ is greater than $2Np'$.
From (1), one can see that
 \begin{align*}
 \phi_{2N}(x)\prod\limits_{i=1}^{k}\phi_{n_i}(x)\prod\limits_{i=1}^{k}\phi_{n'_i}(x)&\equiv \phi_{2N}(x)\prod_{i=1}^{k}(1-x)\prod_{i=1}^{k}(1-x)^{-1}\\
 &\equiv \phi_{2N}(x)(\mod x^{2N}).
 \end{align*} 
 Hence the proposition is true for $m=2k+1$.
 For $m=2k$ consider the product
\begin{equation*}
\phi_{2Np'}(x)\phi_{p'}(x)\prod\limits_{i=1}^{k-1}\phi_{n_i}(x)\prod\limits_{i=1}^{k-1}\phi_{n'_i}(x)\equiv \phi_{2Np'}(x)\phi_{p'}(x)(\mod x^{2Np'}).
\end{equation*}
Hence the proposition is true for $m=2k$ for $k\geq 2$ which completes the proof of the proposition.
\end{proof}
For every natural number $r$ the set $\{|a(r,n)| : n\in\mathbb{N}\}$ remains bounded and let $C(r)$ denote the maximum value of the set. $C(r)$ has been studied by Erd\H{o}s in \cite{I}. Bachman in \cite{J} has obtained the following asymptotic formula
\begin{equation}
\log C(r) = C_o\frac{\sqrt{r}}{(\log r)^\frac{1}{4}}\left(1+O\left(\frac{\log\log r}{\sqrt{\log r}}\right)\right).
\end{equation}
However, this is not true for $r$th coefficients of divisors of $x^n-1$. The $r$th coefficent of divisors of $x^n-1$ are unbounded. In fact more can be said about the first $r$ coefficients of divisors of $x^n-1$.
\begin{theorem 2}\label{theroem 2}
For a given finite sequence of integers $\{n_i\}_{i=1}^{r}$  there exists a divisor $f(x)=\sum\limits_{i=0}^{deg(f)}a_ix^i$ of $x^{l}-1$ for some $l$ such that $a_i=n_i$ for $1\leq i \leq r$.
\end{theorem 2} 
\section{Proof of Theorem \ref{theroem 2}}

The following Lemma is needed for proving Theorem \ref{theroem 2}.

\begin{theorem 7}
For every $n\in \mathbb{N}$, there exist two sequences of polynomials $\{d_m^{(n)}(x)\}_{m=1}^{\infty}$ and $\{d_m^{'(n)}(x)\}_{m=1}^{\infty}$ where $d_m^{(n)}(x)$ and $d_m^{'(n)}(x)$ are divisors of $x^{l_m}-1$ and $x^{l'_m}-1$ respectively for some $l_m,l'_m\in\mathbb{N}$ and have the following properties.
\newline
 (1)$gcd(d_{m_{1}}^{(n)},d_{m_{2}}^{(n)})=1$ for $m_1\neq m_2$.
\newline
(2)$gcd(d_{m_{1}}^{'(n)},d_{m_{2}}^{'(n)})=1$ for $m_1\neq m_2$.\newline
(3)$gcd(d_{m_{1}}^{(n)},d_{m_{2}}^{'(n)})=1$ for all $m_1$ and $m_2$.\newline
(4)$d_m^{(n)}(x)\equiv 1-x^n (\mod x^{n+1})$.\newline
(5)$d_m^{'(n)}(x)\equiv 1+x^n (\mod x^{n+1})$.
\end{theorem 7}
\begin{proof}
From (2)
\begin{equation*}
1-x^n=-\prod_{d|n}\phi_d(x).
\end{equation*}
Let $\{n_i\}_{i=1}^{\infty}$ be a strictly increasing sequence of all square free numbers with exactly two prime factors such that each of the prime divisor is strictly greater than $n$.
From (1), if $1<d\leq n$ then
\begin{equation*}
\phi_{n_id}(x)\equiv \phi_{d}(x)  (\mod x^{n+1}).
\end{equation*}
If $d=1$ then $\phi_{n_id}(x)\equiv -\phi_{1}(x) (\mod x^{n+1})$.
Hence for every $m$
\begin{equation*}
\prod_{d|n}\phi_{n_md}(x)\equiv 1-x^{n} (\mod x^{n+1}).
\end{equation*}
From (2), we have
\begin{equation*}
\frac{(x^{2n}-1)}{x^n-1}=1+x^n= \prod_{\substack {d|2n\\d\nmid n}}\phi_d(x).
\end{equation*}
Thus, we have  
\begin{equation*}
\prod_{\substack{d|2n\\d\nmid n}}\phi_{n_md}(x)\equiv 1+x^{n} (\mod x^{n+1}).
\end{equation*}
Define 
\begin{equation*}
d_m^{(n)}(x):=\prod_{d|n}\phi_{n_md}(x)
\end{equation*}
and \begin{equation*}
d_m^{'(n)}(x):=\prod_{\substack{d|2n\\d\nmid n}}\phi_{n_md}(x).
\end{equation*}
Clearly, $\{d_m^{(n)}(x)\}_{m=1}^{\infty}$ and $\{d_m^{'(n)}(x)\}_{m=1}^{\infty}$ satisfy properties (1) to (5).
\end{proof}
Proof of Theorem \ref{theroem 2}
\begin{proof}
The proof is by induction on $r$.
For $r=1$ and for some $n_1\in \mathbb{Z}$.
If $n_1=0$, let 
$f(x)=d_1^{(2)}(x)\equiv 1+n_1 x (\mod x^2),$ if $n_1>0$ let
$f(x)=\prod\limits_{i=1}^{n_1}d_i^{'(1)}(x)\equiv 1+n_1x(\mod x^2)$ and if $n_1<0$ let
$f(x)=\prod\limits_{i=1}^{-n_1}d_i^{(1)}(x)\equiv 1+n_1x(\mod x^2)$
. Hence the theorem is true for $r=1$.
\newline
Let us assume the theorem is true for $r=k$.
For a sequence $\{n_i\}_{i=1}^{r+1}$ from our assumption there exists a divisor $f'(x)=\sum\limits_{i=0}^{deg(f')}a_ix^i$ of $x^{l'}-1$ for some $l'$ such that $a_i=n_i$ for $1\leq i \leq r$.
\newline
If $a_{r+1}=n_{r+1}$ set $f(x)=f'(x)$.
\newline
If $a_{r+1}>n_{r+1}$ set
\begin{equation*}
f(x)=f'(x)\prod\limits_{i=1}^{a_{r+1}-n_{r+1}}d_{n_{j_i}}^{(r+1)}(x)\equiv a_0+\sum\limits_{i=1}^{r+1}n_ix^i (\mod x^{r+2}),
\end{equation*}
where $d_{n_{j_i}}(x)$ are chosen such that they are relatively prime to $f'(x)$.
\newline
If $a_{r+1}<n_{r+1}$ set
\begin{equation*}
f(x)=f'(x)\prod\limits_{i=1}^{-a_{r+1}+n_{r+1}}d_{n_i}^{'(r+1)}(x)\equiv 1+\sum\limits_{i=1}^{r+1}n_ix^i (\mod x^{r+2}),
\end{equation*}
where $d'_{n_{j_i}}(x)$ are chosen such that they are relatively prime to $f'(x)$. Since the divisors are relatively prime to $f'(x)$, $f(x)$ is a divisor of $x^l-1$ for some $l$.
Hence the theorem is true for $r=k+1$ which completes the proof of the theorem.
\end{proof}
\section{Upper and Lower bounds on $H(r,n)$}
We give an upper bound for $H(r,n)$ in terms of number of divisors of $n$.
\begin{theorem 3}\label{Theorem 3}
For a given natural number $n$ there exists a constant $c(r)$ only depending on $r$ such that
\begin{equation*}
H(r,n)\leq \frac{1}{2^rr!}d(n)^r+c(r)d(n)^{r-1},
\end{equation*}
where $d(n)$ is number of divisors of $n$.
\end{theorem 3}

\begin{proof}
Let $n$ be a natural number. From (2)
\begin{equation*}
x^n-1=\prod_{d|n}\phi_d(x).
\end{equation*}
Any divisor $f(x)\in \mathbb{Z}[x]$ will be of the form 
\begin{equation*}
f(x)=\prod_{m\in S}\phi_m(x),
\end{equation*}
where $S$ is a subset of set of all divisors of $n$.
From (1)
\begin{align*}
f(x)&=\prod_{m\in S}\prod_{d|m}(x^d-1)^{\mu(\frac{m}{d})}\\
&=\prod_{d|n}(x^d-1)^{s_1(d)-s_2(d)}\\
& \equiv \prod_{\substack {d|n \\ d\leq r} }(x^d-1)^{s_1(d)-s_2(d)} (\mod x^{r+1}), 
\end{align*}
where \begin{align*}
& s_1(d):=|\{m\in S : d|m,\mu(\frac{m}{d})=1\}|,\\
& s_2(d):=|\{m\in S : d|m,\mu(\frac{m}{d})=-1\}|.
\end{align*}
Since $|s_1(d)-s_2(d)|\leq \frac{d(n)}{2}$, the coefficients of $f(x) \mod x^{r+1}$ are dominated by 
\begin{align*}
g(x)&=(\prod_{i=1}^{r}(1-x^i)^{-\frac{1}{2}} )^{d(n)}\\
&\equiv (\sum\limits_{i=0}^{r}c_ix^i)^{d(n)}(\mod x^{r+1}),
\end{align*} 
where $c_0=1$ and $c_1=\frac{1}{2}$.
When $(\sum\limits_{i=0}^{r}c_ix^i)^{d(n)}$ is expanded using multinomial theorem, coefficient of $x^r$ will be 
\begin{align*}
((\sum\limits_{i=0}^{r}c_ix^i)^{d(n)})_r&=\sum_{
i_0+\cdots+i_r=d(n)\atop i_1+2i_2+\cdots+ri_r=r}\frac{d(n)!c_1^{i_1}c_2^{i_2}\cdots c_r^{i_r}}{(d(n)-i_1-i_2-\cdots-i_r)!i_1!\cdots i_r!}\\
&=\frac{1}{2^rr!}d(n)^r+O(d(n)^{r-1}).
\end{align*}
Hence 
\begin{equation*}
(f(x))_r\leq\frac{1}{2^rr!}d(n)^r+c(r)d(n)^{r-1},
\end{equation*}
for some constant $c(r)>0$.
Thus,
$H(r,n)\leq \frac{1}{2^rr!}d(n)^r+c(r)d(n)^{r-1}$
for all $n\in \mathbb{N}$.
\end{proof}
 An immediate consequence of this theorem is the following corollary.
\begin{corollary}
 \begin{equation*}
 H(r,n)\leq (1+o(1))n^{\frac{r(\log 2+o(1)}{\log \log n})}.
 \end{equation*} 
 \end{corollary}
 \begin{proof}
 This follows from the theorem of Ramanujan \cite{K} that
 \begin{equation*}
 d(n)\leq n^{\frac{(\log 2+o(1)}{\log \log n})}.
 \end{equation*}
 \end{proof}
 
 Now we show that the inequality can be reversed for infinitely many $n$.
 \begin{theorem 4}\label{Theorem 4}
 There exists a sequence of natural numbers $\{n_k\}_{k=1}^{\infty}$  such that 
 \begin{equation*}
 H(r,n_k)\geq \frac{1}{2^rr!}d(n_k)^r-c_1(r)d(n_k)^{r-1}
 \end{equation*}
 and $d(n_k)\to \infty$ as $k\to \infty$ where $c_1(r)$ is a constant only depending on $r$.
 \end{theorem 4}
 \begin{proof}
 Let $n_k$ be the product of first $k$ primes with $k\geq r$.
 Let \begin{align*}
 f_k(x)&=\prod_{\substack{m|n_k \\ \mu(m)=1}}\phi_m(x)\\
 &=\prod_{\substack{m|n_k\\\mu(m)=1}}\prod_{d|m}(x^d-1)^{\mu(\frac{m}{d})}\\
 &=\prod_{\substack{m|n_k\\\mu(m)=1}}\prod_{d|m}(x^d-1)^{\mu(d)}\\
 &=\prod_{d|n_k}\prod_{\substack{m\equiv 0 (\mod d)\\\mu(m)=1}}(x^d-1)^{\mu(d)}.
 \end{align*}
 Since $|\{m|n_k : m\equiv 0 (\mod d),\mu(m)=1\}|=2^{k-v(d)-1}$ for $d|n_k$ and $d\neq n_k$, 
 where $v(d)$ denotes number of distinct prime factors of $d$, we have
 \begin{align*}
 f_k(x)&\equiv \prod_{d\leq r}(x^d-1)^{\mu(d)2^{k-v(d)-1}} (\mod x^{r+1})\\
 &\equiv\left(\prod_{d\leq r}(1-x^d)^{\mu(d)2^{-v(d)-1}}\right)^{2^k} (\mod x^{r+1})\\
 &\equiv (\sum\limits_{i=0}^{r}c_ix^i)^{d(n_k)}(\mod x^{r+1}),
 \end{align*}
 where $c_0=1$ and $c_1=-\frac{1}{2}$.
 Expanding $f_k(x)$ using multinomial theorem we can arrive at
 \begin{align*}
 ((\sum\limits_{i=0}^{r}c_ix^i)^{d(n)})_r&=\sum_{
 i_0+\cdots+i_r=d(n_k)\atop i_1+2i_2+\cdots+ri_r=r}\frac{d(n_k)!c_1^{i_1}c_2^{i_2}\cdots c_r^{i_r}}{(d(n)-i_1-i_2-\cdots-i_r)!i_1!\cdots i_r!}\\
 &= \frac{(-1)^r}{2^rr!}d(n_k)^r+O(d(n_k)^{r-1}).
 \end{align*}
 Therefore there exists a constant $c_1(r)$ independent of $k$ such that.
 \begin{equation*}
 |(f_k(x))_r|\geq \frac{1}{2^rr!}d(n_k)^r-c_1(r)d(n_k)^{r-1}.
 \end{equation*}
 Thus,
 \begin{equation*}
 H(r,n_k)\geq \frac{1}{2^rr!}d(n_k)^r-c_1(r)d(n_k)^{r-1}.
 \end{equation*}
 \end{proof}
 Since $n_k$ is product of first $k$ primes. We have $d(n_k)=n_k^{\frac{\log 2+\epsilon_k}{\log \log n_k}}$ for some $\epsilon_k$ and $\epsilon_k\rightarrow 0$ as $k\rightarrow \infty$. We have the following corollary. 
 \begin{corollary 1}
For every $r$, there exists a sequence $\{n_k\}_{k=1}^{\infty}$ such that
 \begin{equation*}
 H(r,n_k)\geq (1+\epsilon_k')n_k^{\frac{r(\log 2+\epsilon_k)}{\log \log n_k}},
 \end{equation*}
  as $k\rightarrow \infty$, $\epsilon_k\rightarrow 0$ and $\epsilon_k'\rightarrow 0$.
 \end{corollary 1}
 \bibliographystyle{amsplain}

\begin{thebibliography}{10}
 \bibitem{J}G. Bachman, {\it On the coefficients of cyclotomic polynomial}, Mem. Amer. Math. Soc.,
  106, no. 510 (1993).
 
 \bibitem{B}P. T.  Bateman, {\it Note on the coefficients of the cyclotomic polynomial}, Bull. Amer.
 Math. Soc. {\bf 55} (1949), 1180-1181. MR 0032677 (11,329e)
 
  \bibitem{D}D. M. Bloom, {\it On the coefficients of the cyclotomic polynomials}, Amer. Math. Monthly, 75, 372-377 (1968).
 
 \bibitem{I}P. Erd\H{o}s and R. C. Vaughan, {\it Bounds for the r-th coefficients of cyclotomic polynomials},
  J. London Math. Soc. (2) {\bf 8} (1974), 393-400. MR 0357367 (50 \#9835).
  
 \bibitem{G} N.Kaplan, {\it Bounds for the maximal height of divisors of $x^n-1$}, J. Number
   Theory {\bf 129} (2009), 2673-2688.
 
 \bibitem{E}H. Maier, {\it The size of the coefficients of cyclotomic polynomials, Analytic number
  theory}, Vol. 2 (Allerton Park, IL, 1995), Progr. Math., vol. 139, Birkh¨auser Boston,
  Boston, MA, 1996, pp. 633-639. MR 1409383 (97g:11106)
  
  \bibitem{F}C. Pomerance, N.C. Ryan, {\it Maximal height of divisors of $x^n-1$}, Illinois J. Math. {\bf 51} (2007) 597-604.
 
  \bibitem{K}S. Ramanujan, {\it Highly composite numbers}, Proc. London Math. Soc. (2) {\bf 14} (1915),347-409.
  
 \bibitem{H} N.C. Ryan, B.C. Ward and R. Ward, {\it Some conjectures on the maximal
 height of divisors of $x^n-1$}, Involve {\bf 3} (2010), 451-457.
  
 \bibitem{C}R. C. Vaughan, {\it Bounds for the coefficients of cyclotomic polynomials}, Michigan Math.
  J. {\bf 21} (1974), 289-295 (1975). MR 0364141 (51 \#396)
  
 \bibitem{A}J. Suzuki, {\it On coefficients of cyclotomic polynomials}, Proc. Japan Acad., 63, Ser. A, 279-280 (1987).
   
 \bibitem{Z}L. Thompson, {\it Heights of Divisors of $x^n-1$}, Integers, {\bf 11A} (2011), Article
  20.
  
 \end{thebibliography}

\end{document}